\documentclass{amsart}
\usepackage{amsmath}
  \usepackage{paralist}
  \usepackage{amssymb}
 \usepackage{amsthm}
\usepackage[colorlinks=true]{hyperref}
\hypersetup{urlcolor=blue, citecolor=red}

\newtheorem{theorem}{Theorem}[section]

\newtheorem{lemma}[theorem]{Lemma}

\theoremstyle{definition}

\newtheorem{remark}[theorem]{Remark}

\numberwithin{equation}{section}

\title[Critical Choquard equations]
      {Groundstates for Choquard equations with the upper critical exponent}

\author[Xinfu Li, Shiwang Ma]{}


\thanks{Email Address:  lxylxf@tjcu.edu.cn (XL); shiwangm@nankai.edu.cn
(SM)}

\begin{document}
\maketitle

\centerline{\scshape Xinfu Li$^\mathrm{a}$ and  Shiwang
Ma$^\mathrm{b}$}
\medskip
{\footnotesize
 \centerline{$^\mathrm{a}$School of Science, Tianjin University of Commerce, Tianjin 300134,
P. R. China} \centerline{$^\mathrm{b}$School of Mathematical
Sciences and LPMC, Nankai University, Tianjin 300071, China}}

\bigskip

\begin{abstract}
In this paper, an autonomous Choquard equation with the upper
critical exponent is considered. By using the Poho\v{z}aev
constraint method, the subcritical approximation method and the
compactness lemma of Strauss, a groundstate solution in
$H^1(\mathbb{R}^N)$ which is positive and radially symmetric is
obtained. The result here extends and complements the earlier
theorems.
\end{abstract}

\section{Introduction and main results}

\setcounter{section}{1}
\setcounter{equation}{0}

Let $N\geq 3$ and  $\alpha\in(0,N)$. We are interested in the
autonomous Choquard equation
\begin{equation}\label{e1.1}
-\Delta u+u=(I_{\alpha}\ast G(u))g(u),\quad \mathrm{in}\
\mathbb{R}^N,
\end{equation}
where $g\in C(\mathbb{R},\mathbb{R})$, $G(s)=\int_0^sg(t)dt$, and
$I_{\alpha}$ is the Riesz potential defined for every $x\in
\mathbb{R}^N \setminus \{0\}$ by
\begin{equation}\label{e1.37}
I_{\alpha}(x)=\frac{A_\alpha(N)}{|x|^{N-\alpha}}
\end{equation}
with
$A_\alpha(N)=\frac{\Gamma(\frac{N-\alpha}{2})}{\Gamma(\frac{\alpha}{2})\pi^{N/2}2^\alpha}$
and $\Gamma$ denoting the Gamma function (see \cite{Riesz1949AM},
P.19).

For $G(u)=|u|^p/p^{1/2}$, (\ref{e1.1}) is reduced to the special
equation
\begin{equation}\label{e1.32}
-\Delta u+ u=(I_{\alpha}\ast|u|^{p})|u|^{p-2}u \quad \mathrm{in}\
\mathbb{R}^N.
\end{equation}
When $N=3,\  p=2$ and $\alpha=2$, (\ref{e1.32}) was investigated by
Pekar in \cite{Pekar 1954} to study the quantum theory of a polaron
at rest. In \cite{Lieb 1977}, Choquard applied it as an
approximation to Hartree-Fock theory of one component plasma. It
also arises in multiple particles systems \cite{Gross 1996} and
quantum mechanics \cite{Penrose 1996}. There are many papers devoted
to the existence and  multiplicity of solutions of (\ref{e1.32}) and
their qualitative properties. See the survey paper
\cite{Moroz-Schaftingen 2017} and the references therein. For  $p\in
\left(\frac{N+\alpha}{N},\frac{N+\alpha}{N-2}\right)$, Moroz and
Schaftingen \cite{Moroz-Schaftingen JFA 2013} established the
existence, qualitative properties and decay estimates of
groundstates of (\ref{e1.32}). They also obtain some nonexistence
results under the range
\begin{equation*}
p\geq \frac{N+\alpha }{N-2}\ \text{ or }\ p\leq \frac{N+\alpha }{N}.
\end{equation*}
Usually, $\frac{N+\alpha}{N}$ is called the lower critical exponent
and $\frac{N+\alpha }{N-2}$ is the upper critical exponent for the
Choquard equation.

Compared with (\ref{e1.32}), the nonhomogeneity of the general
nonlinearity $G(s)$ in (\ref{e1.1}) causes much more difficulties.
In spite of this, Moroz and Schaftingen in \cite{Moroz-Schaftingen
2015} made a great progress on equation (\ref{e1.1}) for the
subcritical case. In the spirit of Berestycki and Lions, they
obtained the existence of groundstates  with sufficient and almost
necessary conditions on the nonlinearity $g$. More precisely, they
made the following
assumptions:\\
(g1) There exists $C>0$ such that for every $s\in \mathbb{R}$,
$|sg(s)|\leq  C \left(|s|^{\frac{N+\alpha}{N}}+
|s|^{\frac{N+\alpha}{N-2}}\right)$.\\
(g2) $\lim_{s\to 0}\frac{G(s)}{|s|^{\frac{N+\alpha}{N}}}=0$ and
$\lim_{|s|\to \infty}\frac{G(s)}{|s|^{\frac{N+\alpha}{N-2}}}=0$.\\
(g3) There exists $s_0\in \mathbb{R}\setminus\{0\}$ such that
$G(s_0)\neq0$.\\
(g4) $g$ is odd and has constant sign on $(0,\infty)$.

The main results obtained in \cite{Moroz-Schaftingen 2015} are as
follows.

\smallskip

\textbf{Theorem A1}.  Assume that (g1)-(g3) hold. Then (\ref{e1.1})
has a groundstate in $H^1(\mathbb{R}^N)$. Furthermore, assume that
(g4) holds.  Then (\ref{e1.1}) has a positive groundstate which is
radially symmetric and radially nonincreasing.

\smallskip

\textbf{Theorem A2}.  Assume that (g1) holds. Then every solution
$u\in H^1(\mathbb{R}^N)$ to (\ref{e1.1}) satisfies the Poho\v{z}aev
identity
\begin{equation*}
\frac{N-2}{2}\int_{\mathbb{R}^N}|\nabla u|^2
+\frac{N}{2}\int_{\mathbb{R}^N}|u|^2=\frac{N+\alpha}{2}\int_{\mathbb{R}^N}\left
(I_{\alpha}\ast G(u)\right)G(u).
\end{equation*}

\smallskip

The results in \cite{Moroz-Schaftingen 2015}  can be seen as a
counterpart for Choquard-type equations of the result of Berestycki
and Lions \cite{Berestycki-Lions 1983} which give similar ``almost
necessary" conditions for the existence of a groundstate to the
Schr\"{o}dinger equation
\begin{equation}\label{e1.2}
-\Delta u=h(u),\quad x\in \mathbb{R}^N
\end{equation}
in the subcritical case. We refer the readers to
\cite{Alves-Souto-Montenegro 2012} \cite{Liu-Liao-Tang 2017}
 \cite{Zhang-Zou 2012} \cite{Zhang-Zou 2014} for similar results of
(\ref{e1.2}) in the critical case.

Recently, many authors considered the general Choquard equation
(\ref{e1.1}) in the critical case, see
\cite{Alves-Gao-Squassina-Yang} \cite{Cassani-Zhang 2016} and
\cite{Seok 2018}.

In \cite{Cassani-Zhang 2016},  Cassani and Zhang considered
(\ref{e1.1}) under the following assumptions on $g\in
C(\mathbb{R}^+,\mathbb{R})$:\\
(g5) $\lim_{t\to 0^+}g(t)/t=0$.\\
(g6) $\lim_{t\to +\infty}g(t)/t^{\frac{2+\alpha}{N-2}}=1$.\\
(g7) There exists $\mu>0$ and $q\in (2,(N+\alpha)/(N-2))$ such that
\begin{equation*}
g(t)\geq t^{(2+\alpha)/(N-2)}+\mu t^{q-1},\ t>0.
\end{equation*}

By using a monotonicity trick due to Jeanjean \cite{Jeanjean 1999}
and a suitable decomposition of Palais-Smale sequences, they
obtained the following result.
\smallskip

\textbf{Theorem A3}. Let $N\geq 3$, $\alpha\in ((N-4)_+,N)$ and
\begin{equation}\label{e1.33}
q>\max\{1+\frac{\alpha}{N-2},\frac{N+\alpha}{2(N-2)}\},
\end{equation}
and assume that conditions (g5)-(g7) hold. Then (\ref{e1.1}) admits
a groundstate solution in $H^1(\mathbb{R}^N)$.
\smallskip

For $N=3,\ G(s)=|s|^{3+\alpha}+W(s)$, Alves et al.
\cite{Alves-Gao-Squassina-Yang} considered (\ref{e1.1}) under the
following assumptions on $w\in C(\mathbb{R}^+,\mathbb{R})$, where
$W(s)=\int_0^sw(t)dt$.\\
(w1) There exist $p,\ q,\ \varsigma$ with
\begin{equation*}
(3+\alpha)/3<q\leq p<3+\alpha,\quad 2+\alpha<\varsigma<3+\alpha,
\end{equation*}
and $c_0,\ c_1>0$ such that for all $s\in \mathbb{R}$,
\begin{equation*}
|w(s)|\leq c_0\left(|s|^{q-1}+|s|^{p-1}\right)\quad
\mathrm{and}\quad W(s)\geq c_1|s|^{\varsigma}.
\end{equation*}
(w2) There exists $\mu>2$ such that
\begin{equation*}
0<\mu W(s)\leq 2w(s)s,\quad s\in \mathbb{R}^+.
\end{equation*}
(w3) $w(s)$ is strictly increasing on $\mathbb{R}^+$.

By using Nehari manifold method, they obtained the following result.
\smallskip

\textbf{Theorem A4}. Let $N=3$ and assume that conditions (w1)-(w3)
hold. Then (\ref{e1.1}) admits a groundstate solution in
$H^1(\mathbb{R}^N)$.

\smallskip

Inspired by \cite{Liu-Liao-Tang 2017}, in this paper, we consider
(\ref{e1.1}) with the upper critical exponent. More precisely, we
consider
\begin{equation}\label{e1.16}
-\Delta u+\kappa u=\left(I_{\alpha}\ast
(\mu|u|^{p^*}+F(u))\right)(\mu p^*|u|^{p^*-2}u+f(u)),\quad
\mathrm{in}\ \mathbb{R}^N,
\end{equation}
where $N\geq 3$, $p^*=\frac{N+\alpha}{N-2}$, $\kappa$ and $\mu$ are
positive constants, $f\in C(\mathbb{R},\mathbb{R})$ and
$F(s)=\int_0^sf(t)dt$
satisfy the following conditions:\\
(f1) $f(s)$ is odd and $f(s)\geq 0$ for every $s\geq 0$.\\
(f2) There exists $C>0$ such that for every $s\in \mathbb{R}$,
$|f(s)|\leq  C \left(|s|^{\frac{N+\alpha}{N}-1}+
|s|^{\frac{N+\alpha}{N-2}-1}\right)$.\\
(f3) $\lim_{|s|\to 0}\frac{F(s)}{|s|^{\frac{N+\alpha}{N}}}=0$ and
$\lim_{|s|\to \infty}\frac{F(s)}{|s|^{\frac{N+\alpha}{N-2}}}=0$.\\
(f4)
 \begin{equation*} \left\{\begin{array}{ll}
\lim_{|s|\to \infty}\frac{F(s)}{s^{\frac{N+\alpha-4}{N-2}}}=+\infty,& N\geq 5,\\
\lim_{|s|\to \infty}\frac{F(s)}{s^{\frac{\alpha}{N-2}}|\ln s|}=+\infty,& N=4,\\
\lim_{|s|\to \infty}\frac{F(s)}{s^{1+\alpha}}=+\infty,& N=3.
\end{array}\right.
\end{equation*}

The main result of this paper is as follows.

\begin{theorem}\label{thm1.1}
Assume that $N\geq 3$, $\alpha\in (0,N)$ and (f1)-(f4) hold. Then
for every $\kappa,\ \mu>0$, equation (\ref{e1.16}) admits a positive
groundstate solution  $u\in H^1(\mathbb{R}^N)$ which is radially
symmetric.
\end{theorem}

\begin{remark}\label{rek1.2}
Obviously, Theorem \ref{thm1.1} generalizes Theorem A3 and A4. The
condition (f4) is used for calculating the energy level (see Lemma
\ref {lem shangjie}).
\end{remark}

To prove Theorem \ref{thm1.1}, we use the Poho\v{z}aev identity to
construct a energy level (see (\ref{e1.35})). By using the results
for the subcritical problems established in \cite{Moroz-Schaftingen
2015} we study the limit of the energy level as the exponent
approaches to the critical exponent, and by employing the
compactness lemma of Strauss we establish existence results for the
critical exponent problems. In the process, a subtle estimate of the
energy level for the critical exponent problems is made.

\smallskip

This paper is organized as follows. In Section 2 we give some
preliminaries. Section 3 is devoted to the proof of Theorem
\ref{thm1.1}.

\smallskip

\textbf{Basic notations}: Throughout this paper, we assume $N\geq
3$. $ C_c^{\infty}(\mathbb{R}^N)$ denotes the space of the functions
infinitely differentiable with compact support in $\mathbb{R}^N$.
$L^r(\mathbb{R}^N)$ with $1\leq r<\infty$ denotes the Lebesgue space
with the norms
$\|u\|_r=\left(\int_{\mathbb{R}^N}|u|^r\right)^{1/r}$.
 $ H^1(\mathbb{R}^N)$ is the usual Sobolev space with norm
$\|u\|_{H^1(\mathbb{R}^N)}=\left(\int_{\mathbb{R}^N}|\nabla
u|^2+|u|^2\right)^{1/2}$. $ D^{1,2}(\mathbb{R}^N)=\{u\in
L^{2N/(N-2)}(\mathbb{R}^N): |\nabla u|\in L^2(\mathbb{R}^N)\}$.

\section{Preliminaries and functional framework}

\setcounter{section}{2}
\setcounter{equation}{0}

The following lemma will be frequently used in this paper.

\begin{lemma}\label{lem1.7}
Let $N\geq 3$,  $q\in [2,2N/(N-2)]$ and $u\in H^1(\mathbb{R}^N)$.
Then there exists a positive constant $C$ independent of $q$ and $u$
such that
\begin{equation*}
\|u\|_q\leq C\|u\|_{H^1(\mathbb{R}^N)}.
\end{equation*}
\end{lemma}

\begin{proof}
By the H\"{o}lder inequality and the Sobelev imbedding theorem,
\begin{equation*}
\begin{split}
\|u\|_q\leq \|u\|_{2}^\theta\|u\|_{2N/(N-2)}^{1-\theta}&\leq
(C_1\|u\|_{H^1(\mathbb{R}^N)})^\theta(C_2\|u\|_{H^1(\mathbb{R}^N)})^{1-\theta}\\
&\leq \max\{C_1,C_2\}\|u\|_{H^1(\mathbb{R}^N)},
\end{split}
\end{equation*}
where $\frac{1}{q}=\frac{\theta}{2}+\frac{1-\theta}{2N/(N-2)}$. The
proof is complete.
\end{proof}

The following well known Hardy-Littlewood-Sobolev inequality can be
found in \cite{Lieb-Loss 2001}.

\begin{lemma}\label{lem HLS}
Let $p,\ r>1$ and $0<\alpha<N$ with $1/p+(N-\alpha)/N+1/r=2$. Let
$u\in L^p(\mathbb{R}^N)$ and $v\in L^r(\mathbb{R}^N)$. Then there
exists a sharp constant $C(N,\alpha,p)$, independent of $u$ and $v$,
such that
\begin{equation*}
\left|\int_{\mathbb{R}^N}\int_{\mathbb{R}^N}\frac{u(x)v(y)}{|x-y|^{N-\alpha}}\right|\leq
C(N,\alpha,p)\|u\|_p\|v\|_r.
\end{equation*}
If $p=r=\frac{2N}{N+\alpha}$, then
\begin{equation*}
C(N,\alpha,p)=C_\alpha(N)=\pi^{\frac{N-\alpha}{2}}\frac{\Gamma(\frac{\alpha}{2})}{\Gamma(\frac{N+\alpha}{2})}\left\{\frac{\Gamma(\frac{N}{2})}{\Gamma(N)}\right\}^{-\frac{\alpha}{N}}.
\end{equation*}
\end{lemma}

\begin{remark}\label{rek1.3}
By the Hardy-Littlewood-Sobolev inequality above, for any $v\in
L^s(\mathbb{R}^N)$ with $s\in(1,\frac{N}{\alpha})$, $I_\alpha\ast
v\in L^{\frac{Ns}{N-\alpha s}}(\mathbb{R}^N)$ and
\begin{equation*}
\|I_\alpha\ast v\|_{\frac{Ns}{N-\alpha s}}\leq
A_\alpha(N)C(N,\alpha,s)\|v\|_s,
\end{equation*}
where $A_\alpha(N)$ is defined in (\ref{e1.37}).
\end{remark}

To prove Theorem \ref{thm1.1}, we need the following compactness
lemma of Strauss \cite{Strauss 1977} (see also
\cite{Berestycki-Lions 1983}).

\begin{lemma}\label{lem compact} Let $Q_1,\ Q_2: \mathbb{R}\to \mathbb{R}$
be two continuous functions satisfying
\begin{equation*}
\frac{Q_1(s)}{Q_2(s)}\to 0 \ \text{as}\ |s|\to \infty.
\end{equation*}
Let $\{u_n\}$ be a sequence of measurable functions:
$\mathbb{R}^N\to \mathbb{R}$ such that
\begin{equation*}
\sup_{n}\int_{\mathbb{R}^N}|Q_2(u_n(x))|dx<\infty
\end{equation*}
and $Q_1(u_n(x))\to v(x)$ a.e. in $\mathbb{R}^N$, as $n\to \infty$.
Then for any bounded Borel set $B$ one has
\begin{equation*}
\int_{B}|Q_1(u_n(x))-v(x)|dx\to 0, \ \text{as}\ n\to \infty.
\end{equation*}
If one further assumes that
\begin{equation*}
\frac{Q_1(s)}{Q_2(s)}\to 0 \ \text{as}\ s\to 0
\end{equation*}
and $u_n(x)\to 0$ as $|x|\to \infty$, uniformly with respect to $n$,
then $Q_1(u_n)$ converges to $v$ in $L^1(\mathbb{R}^N)$ as $n\to
\infty$.
\end{lemma}

The following lemma is useful in concerning the uniform bound of
radial nonincreasing functions, see \cite{Berestycki-Lions 1983} for
its proof.

\begin{lemma}\label{lem symmet}
If $u\in L^t(\mathbb{R}^N),\ 1\leq t<+\infty$, is a radial
nonincreasing function (i.e. $0\leq u(x)\leq u(y)$ if $|x|\geq
|y|$), then one has
\begin{equation*}
|u(x)|\leq |x|^{-N/t}\left(\frac{N}{|S^{N-1}|}\right)^{1/t}\|u\|_t,
\ x\neq 0.
\end{equation*}
\end{lemma}

The following lemma can be found in  \cite{Bogachev 2007} and
\cite{Willem 2013}.

\begin{lemma}\label{lem weak}
Let $\Omega\subset \mathbb{R}^N$ be a domain, $q\in (1,\infty)$ and
$\{u_n\}$ be a bounded sequence in $L^{q}(\Omega)$. If $u_n\to u$
a.e. on $\Omega$, then $u_n\rightharpoonup u$ weakly in
$L^q(\Omega)$.
\end{lemma}

To prove Theorem \ref{thm1.1}, inspired by
 \cite{Liu-Liu-Wang 2013}, we introduce
the following equation
\begin{equation}\label{e1.5}
-\Delta u+\kappa u=\left(I_{\alpha}\ast (\mu
|u|^p+F(u))\right)(p\mu|u|^{p-2}u+f(u)),\quad \mathrm{in}\
\mathbb{R}^N,
\end{equation}
where $p\in[p_0,p^*]$, $p_0=(p^*-p_*)/4$,\ $p_*=\frac{N+\alpha}{N}$,
$\kappa$, $\mu$, $p^*$, $\alpha$ and $F(u)$ are defined in
(\ref{e1.16}). For $p=p^*$, equation (\ref{e1.5}) is reduced to
(\ref{e1.16}), and for $p< p^*$, equation (\ref{e1.5}) is
subcritical, which is already considered in \cite{Moroz-Schaftingen
2015}.

Conditions (f1)-(f3), the Hardy-Littlewood-Sobolev inequality and
the Sobolev embedding theorem imply that the functional $I_p:\
H^1(\mathbb{R}^N)\to \mathbb{R}$ defined by
\begin{equation}\label{e1.6}
I_p(u)=\frac{1}{2}\int_{\mathbb{R}^N}|\nabla u|^2
+\kappa|u|^2-\frac{1}{2}\int_{\mathbb{R}^N}\left(I_{\alpha}\ast(\mu|u|^{p}+F(u))\right)(\mu|u|^{p}+F(u))
\end{equation}
is $C^1(H^1(\mathbb{R}^N),\mathbb{R})$ and
\begin{equation}\label{e1.7}
\langle I_p'(u),w\rangle=\int_{\mathbb{R}^N}\nabla u\nabla w +\kappa
uw-\int_{\mathbb{R}^N}\left(I_{\alpha}\ast(\mu|u|^{p}+F(u))\right)(p\mu|u|^{p-2}uw+f(u)w)
\end{equation}
for any $u,\ w\in H^1(\mathbb{R}^N)$. Thus, any critical point of
$I_p$ in $H^{1}(\mathbb{R}^N)$ is a weak solution of (\ref{e1.5}). A
nontrivial solution $u\in H^1(\mathbb{R}^N)$  of (\ref{e1.5}) is
called a groundstate if
\begin{align}\label{e1.34}
I_p(u)=m_p^g:=\inf\{I_p(v):v\in H^{1}(\mathbb{R}^N)\setminus \{0\}\
\mathrm{and}\ I'_p(v)=0 \}.
\end{align}
Furthermore, we define
\begin{equation}\label{e1.35}
m_p=\inf\{I_p(v):v\in H^{1}(\mathbb{R}^N)\setminus \{0\}\
\mathrm{and}\ P_p(v)=0 \},
\end{equation}
where
\begin{equation*}
\begin{split}
P_p(u)=&\frac{N-2}{2}\int_{\mathbb{R}^N}|\nabla u|^2
+\frac{N}{2}\int_{\mathbb{R}^N}\kappa|u|^2\\
&\quad -\frac{N+\alpha}{2}\int_{\mathbb{R}^N}\left(I_{\alpha}\ast
(\mu|u|^{p}+F(u))\right)(\mu|u|^{p}+F(u)).
\end{split}
\end{equation*}

In the following, we will show that there exists $u\in
H^1(\mathbb{R}^N)\setminus \{0\}$ such that $P_p(u)=0$. Thus, $m_p$
is well defined. To this end, we first give an elementary lemma.

\begin{lemma}\label{lem1.1}
Let $a,\ b,\ c,\ \alpha>0$ and  $n\geq 3$  be constants. Define
$f:[0,\infty)\to \mathbb{R}$ as
\begin{align*}
f(t)=a t^{n-2}+b t^{n}-c t^{n+\alpha}.
\end{align*}
Then $f$ has a unique critical point which corresponds to its
maximum.
\end{lemma}

Define $u_\tau$ by
\begin{align}\label{e1.8}
u_{\tau}(x)=\left\{\begin{array}{cc}
u(x/\tau),&\tau>0,\\
0,& \tau=0.
\end{array}
 \right.
\end{align}
Then we have the following result.

\begin{lemma}\label{lem1.2}
Assume that $N\geq 3$, $\alpha\in (0,N)$, $p\in[p_0,p^*]$ and
(f1)-(f3) hold. Then for every $u\in H^{1}(\mathbb{R}^N)\setminus
\{0\}$, there exists a unique $\tau_0>0$ such that
$P_p(u_{\tau_0})=0$. Moreover, $I_p(u_{\tau_0})=\max_{\tau\geq
0}I_p(u_{\tau})$.
\end{lemma}

\begin{proof}
Set $\varphi(\tau)=I_p(u_\tau)$. By direct calculation, we have
\begin{equation}\label{e1.36}
\varphi(\tau)=\frac{\tau^{N-2}}{2}\int_{\mathbb{R}^N}|\nabla
u|^2+\frac{\tau^{N}}{2}\int_{\mathbb{R}^N}\kappa|u|^2-\frac{\tau^{N+\alpha}}{2}\int_{\mathbb{R}^N}(I_{\alpha}\ast
G(u))G(u),
\end{equation}
where $G(u)=F(u)+\mu|u|^p$. By Lemma \ref{lem1.1}, $\varphi(\tau)$
has a unique critical point $\tau_0$ which corresponding to its
maximum. Hence, $I_p(u_{\tau_0})=\max_{\tau\geq 0}I_p(u_{\tau})$ and
\begin{equation*}
\begin{split}
0=\varphi'(\tau_0)=&\frac{N-2}{2}\tau_0^{N-3}\int_{\mathbb{R}^N}|\nabla
u|^2+\frac{N}{2}\tau_0^{N-1}\int_{\mathbb{R}^N}\kappa|u|^2\\
&\quad
-\frac{N+\alpha}{2}\tau_0^{N+\alpha-1}\int_{\mathbb{R}^N}(I_{\alpha}\ast
G(u))G(u).
\end{split}
\end{equation*}
That is, $P_p(u_{\tau_0})=0$. The proof is complete.
\end{proof}

The following remark is about the relationship of $m_p$ and $m_p^g$.

\begin{remark}\label{rek 1.1}
By Theorem A.2, we have $m_{p}\leq m_{p}^g$ for any $p\in[p_0,p^*]$.
Moreover, by the proof of \cite{Moroz-Schaftingen 2015}, we can also
obtain $m_{p}= m_{p}^g$ for any $p\in[p_0,p^*)$. In fact, in
\cite{Moroz-Schaftingen 2015}, they proved that
\begin{equation*}
m_{p}^g=b_p:=\inf_{\gamma\in \Gamma}\sup_{t\in [0,1]}I_p(\gamma(t)),
\end{equation*}
where the set of paths is defined as
\begin{equation*}
\Gamma=\{\gamma\in C([0,1], H^1(\mathbb{R}^N)): \gamma(0)=0,\
I_p(\gamma(1))<0\}.
\end{equation*}
For any $u\in H^1(\mathbb{R}^N)\setminus \{0\}$ with $P_p(u)=0$, let
$u_\tau$ be defined in (\ref{e1.8}). By (\ref{e1.36}), there exists
$\tau_0>0$ large enough such that $I_p(u_{\tau_0})<0$. Lemma
\ref{lem1.2} implies that
\begin{equation*}
b_p\leq \max_{\tau\geq 0}I_p(u_\tau)=I_p(u).
\end{equation*}
Since $u$ is arbitrary, we obtain $m_{p}^g=b_p\leq m_p$. Hence,
$m_{p}^g= m_p$ for any $p\in[p_0,p^*)$.
\end{remark}

\section{Proof of Theorem \ref{thm1.1}}

\setcounter{section}{3}
\setcounter{equation}{0}

In this section, we first give some results about the energy level
$m_p$, and then we prove Theorem \ref{thm1.1}.

\begin{lemma}\label{lem1.8}
Assume that $N\geq 3$, $p\in [p_0,p^*]$ and (f1)-(f3) hold. Then
$m_p\geq 0$.
\end{lemma}

\begin{proof}
Let $\{u_n\}\subset H^1(\mathbb{R}^N)\setminus \{0\}$ be a sequence
satisfying $\lim_{n\to \infty}I_p(u_n)=m_p$ and $P_p(u_n)=0$. Then
we have
\begin{equation*}
\begin{split}
 I_{p}(u_n)&=I_{p}(u_n)-\frac{1}{N+\alpha}P_{p}(u_n)\\
                     &=\left(\frac{1}{2}-\frac{N-2}{2(N+\alpha)}\right)\int_{\mathbb{R}^N}|\nabla
                     u_n|^2+\left(\frac{1}{2}-\frac{N}{2(N+\alpha)}\right)\int_{\mathbb{R}^N}\kappa|u_n|^2\geq 0,
\end{split}
\end{equation*}
which implies that $m_p\geq 0$.
\end{proof}

\begin{lemma}\label{lem1.5}
Assume that $N\geq 3$, $p\in [p_0,p^*)$ and (f1)-(f3) hold. Then
$$\limsup_{p\to p^{*}}m_p\leq m_{p^*}.$$
\end{lemma}

\begin{proof}
For any $\epsilon\in(0,1)$, there exists $u\in
H^1(\mathbb{R}^N)\setminus \{0\}$ with $P_{p^*}(u)=0$ such that
$I_{p^*}(u)<m_{p^*}+\epsilon$. In view of (\ref{e1.36}), there
exists $\tau_0>0$ large enough such that $I_{p^*}(u_{\tau_0})\leq
-2$. By the Young inequality, we have
\begin{equation}\label{e1.60}
|u|^{p}\leq \frac{p^*-p}{p^*-p_*}|u|^{p_*}+
\frac{p-p_*}{p^*-p_*}|u|^{p^*},
\end{equation}
and by (f2), we have
\begin{equation}\label{e1.62}
|F(u)|\leq  C \left(|u|^{p_*}+ |u|^{p^*}\right),
\end{equation}
and by the Hardy-Littlewood-Sobolev inequality and the Sobolev
embedding theorem, there exist $C_1,\ C_2>0$ independent of $u$,
such that
\begin{equation}\label{e1.61}
\begin{split}
&\int_{\mathbb{R}^N}\left(I_{\alpha}\ast
|u|^{p_*}\right)|u|^{p_*}\leq C_1
\|u\|_2^{2p_*}\leq C_2\|u\|_{H^1(\mathbb{R}^N)}^{2p_*},\\
&\int_{\mathbb{R}^N}\left(I_{\alpha}\ast
|u|^{p^*}\right)|u|^{p^*}\leq C_1
\|u\|_{2N/(N-2)}^{2p^*}\leq C_2\|u\|_{H^1(\mathbb{R}^N)}^{2p^*},\\
& \int_{\mathbb{R}^N}\left(I_{\alpha}\ast
|u|^{p_*}\right)|u|^{p^*}\leq C_1
\|u\|_2^{p_*}\|u\|_{2N/(N-2)}^{p^*}\leq
C_2\|u\|_{H^1(\mathbb{R}^N)}^{p_*+p^*}.
\end{split}
\end{equation}
Hence, the Lebesgue dominated convergence theorem implies that
$$\frac{\tau^{N+\alpha}}{2}\int_{\mathbb{R}^N}\left(I_{\alpha}\ast
(\mu|u|^p+F(u))\right)(\mu|u|^p+F(u))$$ is continuous on
$p\in[p_0,p^*]$ uniformly with $\tau\in [0,\tau_0]$. Then there
exists $\delta>0$ such that $
|I_p(u_{\tau})-I_{p^*}(u_{\tau})|<\epsilon $ for  $p^*-\delta<p<p^*$
and $0\leq \tau\leq \tau_0$, which implies that $I_p(u_{\tau_0})\leq
-1$ for all $p^*-\delta<p<p^*$. Since $I_{p}(u_{\tau})>0$ for $\tau$
small enough and $I_p(u_0)=0$ for any $p\in[p_0,p^*]$, there exists
$\tau_p\in (0,\tau_0)$ such that
$\frac{d}{d\tau}I_p(u_{\tau})\mid_{\tau=\tau_p}=0$ and then
$P_p(u_{\tau_p})=0$. By Lemma \ref{lem1.2}, $I_{p^*}(u_{\tau_p})\leq
I_{p^*}(u)$. Hence,
\begin{equation*}
m_p\leq I_p(u_{\tau_p})\leq I_{p^*}(u_{\tau_p})+\epsilon\leq
I_{p^*}(u)+\epsilon<m_{p^*}+2\epsilon
\end{equation*}
for any $p^*-\delta<p<p^*$. Thus, $\limsup_{p\to p^{*}}m_p\leq
m_{p^*}$.
\end{proof}

Let $p_n\in [p_0,p^*)$ be a sequence satisfying  $\lim_{n\to
\infty}p_n=p^{*}$. Theorem A1, Theorem A2 and Remark \ref{rek 1.1}
imply that there exists a positive and radial nonincreasing sequence
$\{u_n\}\subset H_r^1(\mathbb{R}^N)\setminus \{0\}$ such that
\begin{equation} \label{e1.9}
I'_{p_n}(u_n)=0,\  I_{p_n}(u_n)=m_{p_n}\ \mathrm{and}\
P_{p_n}(u_n)=0.
\end{equation}

For such chosen sequence, we have the following result.

\begin{lemma}\label{lem xiajixian}
Assume that $N\geq 3$, (f1)-(f3) hold and  $\{u_n\}\subset
H_r^1(\mathbb{R}^N)\setminus \{0\}$ satisfies (\ref{e1.9}). Then
$\{u_n\}$ is bounded in $H^1(\mathbb{R}^N)$ and $\liminf_{n\to
\infty}m_{p_n}>0$.
\end{lemma}

\begin{proof}
By Lemma \ref{lem1.5}, for $n$ large enough, we have
\begin{equation}\label{e1.11}
\begin{split}
m_{p^*}+1\geq m_{p_n}&=I_{p_n}(u_n)-\frac{1}{N+\alpha}P_{p_n}(u_n)\\
                     &=\left(\frac{1}{2}-\frac{N-2}{2(N+\alpha)}\right)\int_{\mathbb{R}^N}|\nabla
                     u_n|^2+\left(\frac{1}{2}-\frac{N}{2(N+\alpha)}\right)\int_{\mathbb{R}^N}\kappa|u_n|^2,
\end{split}
\end{equation}
which implies that $\{u_n\}$ is bounded in $H^1(\mathbb{R}^N)$.

In view of (\ref{e1.60})-(\ref{e1.61}), and by the Cauchy
inequality, there exists $C_3,\ C_4>0$ independent of $n$, such that
\begin{equation*}
\begin{split}
0=P_{p_n}(u_n)&=\frac{N-2}{2}\int_{\mathbb{R}^N}|\nabla u_n|^2
+\frac{N}{2}\int_{\mathbb{R}^N}\kappa|u_n|^2\\
&\quad\quad
-\frac{N+\alpha}{2}\int_{\mathbb{R}^N}\left(I_{\alpha}\ast
(\mu|u_n|^{p_n}+F(u_n))\right)(\mu|u_n|^{p_n}+F(u_n))\\
&\geq
C_3\|u_n\|_{H^1(\mathbb{R}^N)}^{2}-C_4\left(\|u_n\|_{H^1(\mathbb{R}^N)}^{2p_n}+\|u_n\|_{H^1(\mathbb{R}^N)}^{2\frac{N+\alpha}{N}}+\|u_n\|_{H^1(\mathbb{R}^N)}^{2\frac{N+\alpha}{N-2}}\right),
\end{split}
\end{equation*}
which implies that there exists $C_5>0$ independent of $n$, such
that
\begin{equation}\label{e1.12}
\|u_n\|_{H^1(\mathbb{R}^N)}\geq C_5.
\end{equation}
Combining (\ref{e1.11}) and (\ref{e1.12}), we obtain that
$\liminf_{n\to \infty}m_{p_n}>0$.
\end{proof}

By Lemma \ref{lem1.5} and \ref{lem xiajixian}, we have $m_{p^*}>0$.
In the following, we give an upper estimate of $m_{p^*}$.

\begin{lemma}\label{lem shangjie}
Assume that $N\geq 3$ and (f1)-(f4) hold. Then
$$m_{p^*}<\frac{2+\alpha}{2(N+\alpha)}\left(\frac{N-2}{N+\alpha}\right)^{\frac{N-2}{2+\alpha}}\mu^{-\frac{2(N-2)}{2+\alpha}}S_\alpha^{\frac{N+\alpha}{2+\alpha}}.$$
\end{lemma}

\begin{proof}
Let $\varphi \in C_c^{\infty}(\mathbb{R}^N)$ be a cut off function
satisfying: (a) $0\leq \varphi(x)\leq 1$ for any $x\in
\mathbb{R}^N$; (b) $\varphi(x)\equiv 1$ in $B_1$; (c)
$\varphi(x)\equiv 0$ in $\mathbb{R}^N\setminus \overline{B_2}$.
Here, $B_a$ denotes the ball in $\mathbb{R}^N$ of center at origin
and radius $a$. For any $\epsilon>0$, we define
$u_\epsilon(x)=\varphi(x)U_\epsilon(x)$, where
\begin{equation*}
U_\epsilon(x)=\frac{\left(N(N-2)\epsilon^2\right)^{\frac{N-2}{4}}}{\left(\epsilon^2+|x|^2\right)^{\frac{N-2}{2}}}.
\end{equation*}
By \cite{Brezis-Nirenberg 1983} (see also \cite{Willem 1996}), we
have the following estimates.
\begin{equation}\label{e1.18}
\int_{\mathbb{R}^N}|\nabla
u_\epsilon|^2=S^{\frac{N}{2}}+O(\epsilon^{N-2}),\ N\geq 3,
\end{equation}
\begin{equation}\label{e1.19}
\int_{\mathbb{R}^N}|
u_\epsilon|^{2N/(N-2)}=S^{\frac{N}{2}}+O(\epsilon^N),\ N\geq 3,
\end{equation}
and
\begin{equation}\label{e1.20}
\int_{\mathbb{R}^N}| u_\epsilon|^2=\left\{\begin{array}{ll}
K_2\epsilon^2+O(\epsilon^{N-2}),& N\geq 5,\\
K_2\epsilon^2|\ln \epsilon|+O(\epsilon^2),& N=4,\\
K_2\epsilon+O(\epsilon^2),& N=3,
\end{array}\right.
\end{equation}
where $K_1,\ K_2>0$, and
\begin{equation*}
S:=\inf_{ u\in
D^{1,2}(\mathbb{R}^N)\setminus\{0\}}\frac{\int_{\mathbb{R}^N}|\nabla
u|^2}{\left(\int_{\mathbb{R}^N}|u|^{\frac{2N}{N-2}}\right)^{\frac{N-2}{N}}}.
\end{equation*}
Moreover, similar as in \cite{Gao-Yang-1} and \cite{Gao-Yang-2}, by
direct computation,
\begin{equation}\label{e1.24}
\int_{\mathbb{R}^N}\left(I_\alpha\ast|u_\epsilon|^{p^*}\right)
|u_\epsilon|^{p^*}\geq (A_\alpha(N)
C_\alpha(N))^{\frac{N}2}S_\alpha^{\frac{N+\alpha}2}
+O(\epsilon^{\frac{N+\alpha}{2}}),
\end{equation}
where
\begin{equation*}
\begin{split}
S_\alpha:&=\inf_{ u\in
D^{1,2}(\mathbb{R}^N)\setminus\{0\}}\frac{\int_{\mathbb{R}^N}|\nabla
u|^2}{\left(\int_{\mathbb{R}^N}\left(I_{\alpha}\ast
|u|^{\frac{N+\alpha}{N-2}}\right)|u|^{\frac{N+\alpha}{N-2}}\right)^{\frac{N-2}{N+\alpha}}}\\
&=\frac{S}{(A_\alpha(N) C_\alpha(N))^{\frac{N-2}{N+\alpha}}},
\end{split}
\end{equation*}
$A_\alpha(N)$ is defined in (\ref{e1.37}) and $C_\alpha(N)$ is in
Lemma \ref{lem HLS}.

In the following, we use $u_\epsilon$ to estimate $m_{p^*}$. By
Lemma \ref{lem1.2}, there exists a unique $\tau_\epsilon$ such that
$P_{p^*}((u_\epsilon)_{\tau_\epsilon})=0$ and
$I_{p^*}((u_\epsilon)_{\tau_\epsilon})=\sup_{\tau\geq
0}I_{p^*}((u_\epsilon)_{\tau})$. Thus,
\begin{equation}\label{e1.17}
\begin{split}
m_{p^*}&\leq \sup_{\tau\geq 0}I_{p^*}((u_\epsilon)_{\tau})\\
       &=\sup_{\tau\geq 0}\left\{\frac{\tau^{N-2}}{2}\int_{\mathbb{R}^N}|\nabla
u_\epsilon|^2+\frac{\tau^{N}}{2}\int_{\mathbb{R}^N}\kappa|u_\epsilon|^2\right.\\
&\quad\quad\quad\quad
\left.-\frac{\tau^{N+\alpha}}{2}\int_{\mathbb{R}^N}\left(I_{\alpha}\ast
(\mu|u_\epsilon|^{p^*}+F(u_\epsilon))\right)(\mu|u_\epsilon|^{p^*}+F(u_\epsilon))\right\}
\end{split}
\end{equation}

We claim that there exist $\tau_0,\ \tau_1>0$ independent of
$\epsilon$ such that $\tau_\epsilon\in [\tau_0, \tau_1]$ for
$\epsilon>0$ small. Suppose by contradiction that $\tau_\epsilon\to
0$ or $\tau_\epsilon\to \infty$ as $\epsilon\to 0$. $F(s)\geq 0$,
(\ref{e1.18}), (\ref{e1.20}), (\ref{e1.24}) and (\ref{e1.17}) imply
that $m_{p^*}\leq  0$ as $\epsilon\to 0$, which contradicts
$m_{p^*}>0$. Thus, the claim holds.

Set
\begin{equation}\label{e1.25}
J_\epsilon(\tau)=\frac{\tau^{N-2}}{2}\int_{\mathbb{R}^N}|\nabla
u_\epsilon|^2-\frac{\tau^{N+\alpha}}{2}\mu^2\int_{\mathbb{R}^N}\left(I_{\alpha}\ast
|u_\epsilon|^{p^*}\right)|u_\epsilon|^{p^*}.
\end{equation}
In view of (\ref{e1.18}) and (\ref{e1.24}), and by direct
calculation, we have
\begin{equation}\label{e1.26}
\begin{split}
J_\epsilon(\tau)&\leq
\frac{2+\alpha}{2(N+\alpha)}\left(\frac{N-2}{N+\alpha}\right)^{\frac{N-2}{2+\alpha}}\mu^{-\frac{2(N-2)}{2+\alpha}}\frac{\left(\int_{\mathbb{R}^N}|\nabla
u_\epsilon|^2\right)^{\frac{N+\alpha}{2+\alpha}}}{\left(\int_{\mathbb{R}^N}\left(I_{\alpha}\ast
|u_\epsilon|^{p^*}\right)|u_\epsilon|^{p^*}\right)^{\frac{N-2}{2+\alpha}}}\\
&\leq \left\{\begin{array}{ll}
A\left(1+O(\epsilon^{N-2})+O(\epsilon^{\frac{N+\alpha}{2}})\right),& N\geq 4,\\
A\left(1+O(\epsilon)\right),& N=3,
\end{array}\right.
\end{split}
\end{equation}
where
\begin{equation*}
\begin{split}
A&=\frac{2+\alpha}{2(N+\alpha)}\left(\frac{N-2}{N+\alpha}\right)^{\frac{N-2}{2+\alpha}}\mu^{-\frac{2(N-2)}{2+\alpha}}\frac{S^{\frac{N(N+\alpha)}{2(2+\alpha)}}}{(A_\alpha(N)
C_\alpha(N))^{\frac{N(N-2)}{2(2+\alpha)}}S_\alpha^{\frac{(N+\alpha)(N-2)}{2(2+\alpha)}}}\\
&=\frac{2+\alpha}{2(N+\alpha)}\left(\frac{N-2}{N+\alpha}\right)^{\frac{N-2}{2+\alpha}}\mu^{-\frac{2(N-2)}{2+\alpha}}S_\alpha^{\frac{N+\alpha}{2+\alpha}}.
\end{split}
\end{equation*}

By the definition of $u_\epsilon$, for $|x|<\epsilon$ with
$\epsilon<1$, we have
\begin{equation}\label{e1.28}
u_\epsilon(x)\geq \left(\frac{N(N-2)}{4}\right)^{\frac{N-2}{4}}
\epsilon^{-\frac{N-2}{2}}.
\end{equation}
Condition (f4) implies that, for any $R>0$, there exists $S_R>0$
such that for $s\in[S_R,\infty)$,
\begin{equation}\label{e1.29}
F(s)\geq \left\{\begin{array}{ll}
R s^{\frac{N+\alpha-4}{N-2}},& N\geq 5,\\
R s^{\frac{\alpha}{N-2}}|\ln s|,& N=4,\\
R s^{1+\alpha},& N=3.
\end{array}\right.
\end{equation}
In view of $F(s)\geq 0$, (\ref{e1.28}) and (\ref{e1.29}), for any
$R>0$, there exists $\epsilon_0>0$ small enough such that for any
$\epsilon\in (0,\epsilon_0)$,
\begin{equation}\label{e1.30}
\begin{split}
\int_{\mathbb{R}^N}\left(I_{\alpha}\ast
|u_\epsilon|^{p^*}\right)F(u_\epsilon) &\geq
A_\alpha(N)\int_{B_\epsilon}\int_{B_\epsilon}\frac{1}{|x-y|^{N-\alpha}}|u_\epsilon(x)|^{p^*}F(u_\epsilon(y))dxdy\\
&\geq
A_\alpha(N)\int_{B_\epsilon}\int_{B_\epsilon}\frac{1}{|2\epsilon|^{N-\alpha}}\frac{\left(N(N-2)\epsilon^2\right)^{\frac{N+\alpha}{4}}}{\left(2\epsilon^2\right)^{\frac{N+\alpha}{2}}}F(u_\epsilon(y))dxdy\\
&\geq \left\{\begin{array}{ll}
CR \epsilon^{2},& N\geq 5,\\
CR \epsilon^{2}|\ln \epsilon|,& N=4,\\
CR \epsilon,& N=3,
\end{array}\right.
\end{split}
\end{equation}
where $C>0$ is a constant independent of $\epsilon$ and $R$.

Inserting (\ref{e1.20}), (\ref{e1.26}), (\ref{e1.30}) into
(\ref{e1.17}) and by choosing $R$ large enough, we get that
\begin{equation}\label{e1.31}
\begin{split}
m_{p^*}<\frac{2+\alpha}{2(N+\alpha)}\left(\frac{N-2}{N+\alpha}\right)^{\frac{N-2}{2+\alpha}}\mu^{-\frac{2(N-2)}{2+\alpha}}S_\alpha^{\frac{N+\alpha}{2+\alpha}}.
\end{split}
\end{equation}
\end{proof}

\textbf{Proof of Theorem \ref{thm1.1}}. Let $\{u_n\}\subset
H_r^1(\mathbb{R}^N)$ be a positive and radial nonincreasing sequence
which satisfies (\ref{e1.9}) with $p_n\to p^{*-}$ as $n\to \infty$.
Lemma \ref{lem xiajixian} shows that $\{u_n\}$ is bounded in
$H^1(\mathbb{R}^N)$. Thus, there exists a nonnegative  function
$u\in H_r^1(\mathbb{R}^N)$ such that up to a subsequence,
$u_n\rightharpoonup u$ weakly in $H^1(\mathbb{R}^N)$ and $u_n\to u$
a.e. in $\mathbb{R}^N$. By Lemma \ref{lem1.7}, (f2), $p_n\to p^{*-}$
and the H\"{o}lder inequality, we have
\begin{equation}\label{e1.63}
\{\mu|u_n|^{p_n}+F(u_n)\}\ \textrm{is\ bounded\ in}\
L^{\frac{2N}{N+\alpha}}(\mathbb{R}^N),
\end{equation}
$\{p_n\mu|u_n|^{p_n-2}u_n\}$ is bounded in
$L^{\frac{2Np^*}{(p^*-1)(N+\alpha)}}(\mathbb{R}^N)$,
$\{p_n\mu|u_n|^{p_n-2}u_n\psi\}$ is bounded in
$L^{\frac{2N}{N+\alpha}}(\mathbb{R}^N)$ and
$\{p^*\mu|u|^{p^*-2}u\psi\}\in
L^{\frac{2N}{N+\alpha}}(\mathbb{R}^N)$ for any $\psi\in
C_c^{\infty}(\mathbb{R}^N)$. By (\ref{e1.63}) and Lemma \ref{lem
weak}, we have
\begin{equation}\label{e1.65}
\mu|u_n|^{p_n}+F(u_n)\rightharpoonup  \mu|u|^{p^*}+F(u) \
\textrm{weakly\ in}\ L^{\frac{2N}{N+\alpha}}(\mathbb{R}^N).
\end{equation}
Since $\{u_n\}$ is bounded in $L^2(\mathbb{R}^N)\cap
L^{2N/(N-2)}(\mathbb{R}^N)$ and $|f(s)|\leq C(|s|^{p_*-1}+
|s|^{p^*-1})$, by Theorem A.4 in \cite{Willem 1996}, $\{f(u_n)\}$ is
bounded in $L^{\frac{2N}{\alpha}}(\mathbb{R}^N)+
L^{\frac{2N}{2+\alpha}}(\mathbb{R}^N)$. Then the H\"{o}lder
inequality implies that $\{f(u_n)\psi\}$ is bounded in
$L^{\frac{2N}{N+\alpha}}(\mathbb{R}^N)$ for any $\psi\in
C_c^{\infty}(\mathbb{R}^N)$. Similarly, $f(u)\psi\in
L^{\frac{2N}{N+\alpha}}(\mathbb{R}^N)$. By Remark \ref{rek1.3},
\begin{equation}\label{e1.64}
I_\alpha\ast(p^*\mu|u|^{p^*-2}u\psi+f(u)\psi)\in
L^{\frac{2N}{N-\alpha}}(\mathbb{R}^N).
\end{equation}
By (\ref{e1.65}) and (\ref{e1.64}),
\begin{equation*}
\begin{split}
\int_{\mathbb{R}^N}\left(I_{\alpha}\ast(\mu|u_n|^{p_n}+F(u_n))\right)(p^*\mu|u|^{p^*-2}u\psi+f(u)\psi)\\
\to\int_{\mathbb{R}^N}\left(I_{\alpha}\ast(\mu|u|^{p^*}+F(u))\right)(p^*\mu|u|^{p^*-2}u\psi+f(u)\psi).
\end{split}
\end{equation*}
It follows from $N\geq 3$ that $\frac{N}{\frac{N-2}2(p_*-1)}$ and
$\frac{N}{\frac{N-2}2(p^*-1)} \in (\frac{2N}{N+\alpha},\infty)$.
Since  $p_n\to p^{*-}$, $\psi \in L^{r}(\mathbb{R}^N)$ for $r\in
(1,\infty)$, by (f2), the Young inequality, the H\"{o}lder
inequality and Lemma \ref{lem symmet} with $t=2N/(N-2)$, there
exists a constant $C>0$ independent of $n$ such that
\begin{equation}\label{e1.45}
\begin{split}
\left||u_n|^{p_n-2}u_n\psi\right|,&\ \left|f(u_n)\psi\right|\leq
C\left(|u_n|^{p_*-1}|\psi|+|u_n|^{p^*-1}|\psi|\right)\\
&\leq
C\left(|x|^{\frac{2-N}2(p_*-1)}|\psi|+|x|^{\frac{2-N}2(p^*-1)}|\psi|\right)\in
L^{\frac{2N}{N+\alpha}}(\mathbb{R}^N).
\end{split}
\end{equation}
By the Hardy-Littlewood-Sobolev inequality and the Lebesgue
dominated convergence theorem,
\begin{equation*}
\begin{split}
&\left|\int_{\mathbb{R}^N}\left(I_{\alpha}\ast(\mu|u_n|^{p_n}+F(u_n))\right)(p_n\mu|u_n|^{p_n-2}u_n\psi+f(u_n)\psi)\right.\\
&\quad\quad\quad\left.-\int_{\mathbb{R}^N}\left(I_{\alpha}\ast(\mu|u_n|^{p_n}+F(u_n))\right)(p^*\mu|u|^{p^*-2}u\psi+f(u)\psi)\right|=o_n(1).
\end{split}
\end{equation*}
Thus, for any $\psi\in C_c^{\infty}(\mathbb{R}^N)$,
\begin{equation*}
\begin{split}
0=\langle I_{p_n}'(u_n),\psi\rangle &=\int_{\mathbb{R}^N}\nabla
u_n\nabla \psi
+\kappa u_n\psi\\
&\quad\quad -\int_{\mathbb{R}^N}\left(I_{\alpha}\ast(\mu|u_n|^{p_n}+F(u_n))\right)(p_n\mu|u_n|^{p_n-2}u_n\psi+f(u_n)\psi)\\
& \to \int_{\mathbb{R}^N}\nabla u\nabla \psi +\kappa
u\psi\\
&\quad\quad-\int_{\mathbb{R}^N}\left(I_{\alpha}\ast(\mu|u|^{p^*}+F(u))\right)(p^*\mu|u|^{p^*-2}u\psi+f(u)\psi)
\end{split}
\end{equation*}
as $n\to \infty$. That is, $u$ is a solution of (\ref{e1.16}).

We claim that $u\not\equiv 0$. Suppose by contradiction that
$u\equiv0$. Set $Q(s)=|s|^{\frac{N+\alpha}{N}}+
|s|^{\frac{N+\alpha}{N-2}}$. By (f1)-(f3), we have
\begin{equation*}
\lim_{|s|\to 0}\frac{F(s)}{Q(s)}=0 \quad \mathrm{and} \quad
\lim_{|s|\to \infty}\frac{F(s)}{Q(s)}=0.
\end{equation*}
Obviously, $\int_{\mathbb{R}^N}|Q(u_n(x))|^{\frac{2N}{N+\alpha}}\leq
C$ and $F(u_n(x))\to 0$ a.e. in $\mathbb{R}^N$. Moreover, Lemma
\ref{lem symmet} implies that $u_n(x)\to 0$ as $|x|\to \infty$
uniformly with respect to $n$. Then Lemma \ref{lem compact} implies
that $\int_{\mathbb{R}^N}|F(u_n(x))|^{\frac{2N}{N+\alpha}}\to 0$ as
$n\to \infty$. By Lemma \ref{lem1.7} and \ref{lem HLS}, we have
\begin{equation}\label{e1.40}
\int_{\mathbb{R}^N}(I_{\alpha}\ast F(u_n))F(u_n)\leq
C\|F(u_n)\|_{\frac{2N}{N+\alpha}}^2=o(1),
\end{equation}
\begin{equation}\label{e1.41}
\int_{\mathbb{R}^N}(I_{\alpha}\ast |u_n|^{p_n})F(u_n)\leq
C\|F(u_n)\|_{\frac{2N}{N+\alpha}}\|u_n\|_{\frac{2Np_n}{N+\alpha}}^{p_n}=o(1).
\end{equation}
In view of (\ref{e1.40}) and (\ref{e1.41}), and by using
$P_{p_n}(u_n)=0$ and the Young inequality
\begin{equation}\label{e1.14}
|u|^{p_n}\leq \frac{p^*-p_n}{p^*-p_*}|u|^{p_*}+
\frac{p_n-p_*}{p^*-p_*}|u|^{p^*},
\end{equation}
we get that
\begin{equation}\label{e1.13}
\begin{split}
&\int_{\mathbb{R}^N}|\nabla u_n|^2
+\frac{N}{N-2}\int_{\mathbb{R}^N}\kappa|u_n|^2\\
&\quad \quad
=\frac{N+\alpha}{N-2}\int_{\mathbb{R}^N}\left(I_{\alpha}\ast
(\mu|u_n|^{p_n}+F(u_n))\right)(\mu|u_n|^{p_n}+F(u_n))\\
&\quad
\quad=\frac{N+\alpha}{N-2}\mu^2\int_{\mathbb{R}^N}\left(I_{\alpha}\ast
|u_n|^{p_n}\right)|u_n|^{p_n}+o(1)\\
&\quad \quad\leq
p^*\mu^2\left(\frac{p_n-p_*}{p^*-p_*}\right)^2\int_{\mathbb{R}^N}(I_{\alpha}\ast
|u_n|^{p^*})|u_n|^{p^*}+o(1)\\
&\quad \quad=p^*\mu^2\int_{\mathbb{R}^N}(I_{\alpha}\ast
|u_n|^{p^*})|u_n|^{p^*}+o(1)\\
&\quad \quad\leq p^*\mu^2\left(\frac{\int_{\mathbb{R}^N}|\nabla
u_n|^2}{S_\alpha}\right)^{p^*}+o(1)
\end{split}
\end{equation}
which implies that either $\|u_n\|_{H^1(\mathbb{R}^N)}\to 0$ or
$$\limsup_{n\to \infty}\|\nabla u_n\|_2^2\geq \left(p^*\mu^2
S_\alpha^{-p^*}\right)^{-\frac{N-2}{2+\alpha}}.$$ If
$\|u_n\|_{H^1(\mathbb{R}^N)}\to 0$, then (\ref{e1.11}) implies that
$m_{p_n}\to 0$, which contradicts Lemma \ref{lem xiajixian}.\\
If $\limsup_{n\to \infty}\|\nabla u_n\|_2^2\geq \left(p^*\mu^2
S_\alpha^{-p^*}\right)^{-\frac{N-2}{2+\alpha}}$, then
\begin{equation*}
\begin{split}
m_{p^*}&\geq \limsup_{n\to \infty}m_{p_n}\\
       &=\limsup_{n\to\infty}\left(I_{p_n}(u_n)-\frac{1}{N+\alpha}P_{p_n}(u_n)\right)\\
       &=\limsup_{n\to\infty}\left(\frac{1}{2}-\frac{N-2}{2(N+\alpha)}\right)\int_{\mathbb{R}^N}|\nabla
       u_n|^2+\left(\frac{1}{2}-\frac{N}{2(N+\alpha)}\right)\int_{\mathbb{R}^N}\kappa|
       u_n|^2\\
       &\geq \left(\frac{1}{2}-\frac{N-2}{2(N+\alpha)}\right)\left(p^*\mu^2
S_\alpha^{-p^*}\right)^{-\frac{N-2}{2+\alpha}}\\
&=\frac{2+\alpha}{2(N+\alpha)}\left(\frac{N-2}{N+\alpha}\right)^{\frac{N-2}{2+\alpha}}\mu^{-\frac{2(N-2)}{2+\alpha}}S_\alpha^{\frac{N+\alpha}{2+\alpha}},
\end{split}
\end{equation*}
which contradicts Lemma \ref{lem shangjie}. Thus $u\not\equiv0$. By
Theorem A.2, $P_{p^*}(u)=0$.

By the weakly lower semi-continuity of the norm, we have
\begin{equation*}
\begin{split}
m_{p^*}&\leq I_{p^*}(u)\\
       &=\left(I_{p^*}(u)-\frac{1}{N+\alpha}P_{p^*}(u)\right)\\
        &=\left(\frac{1}{2}-\frac{N-2}{2(N+\alpha)}\right)\int_{\mathbb{R}^N}|\nabla
       u|^2+\left(\frac{1}{2}-\frac{N}{2(N+\alpha)}\right)\int_{\mathbb{R}^N}\kappa|
       u|^2\\
       &\leq\liminf_{n\to\infty}\left(\frac{1}{2}-\frac{N-2}{2(N+\alpha)}\right)\int_{\mathbb{R}^N}|\nabla
       u_n|^2+\left(\frac{1}{2}-\frac{N}{2(N+\alpha)}\right)\int_{\mathbb{R}^N}\kappa|
       u_n|^2\\
       &=\liminf_{n\to\infty}\left(I_{p_n}(u_n)-\frac{1}{N+\alpha}P_{p_n}(u_n)\right) \\
       &=\liminf_{n\to\infty}m_{p_n}\leq \limsup_{n\to\infty}m_{p_n}\leq m_{p^*}.
\end{split}
\end{equation*}
Hence $I_{p^*}(u)=m_{p^*}$. By the definition of $m_{p^*}^g$, we
have $m_{p^*}^g\leq I_{p^*}(u)=m_{p^*}$, which combining with Remark
\ref{rek 1.1} show that $m_{p^*}^g=m_{p^*}=I_{p^*}(u)$. That is, $u$
is a nonnegative and radial groundstate solution of (\ref{e1.16}).
The strongly maximum principle implies that $u$ is positive. The
proof is complete.

\textbf{Acknowledgements.} This work was supported by the research
project of Tianjin education commission with the Grant no. 2017KJ173
``Qualitative studies of solutions for two kinds of nonlocal
elliptic equations" and the National Natural Science Foundation of
China (No. 11571187).


\begin{thebibliography}{00}
%
%
\bibitem{Alves-Gao-Squassina-Yang} C.O. Alves, F. Gao, M. Squassina, M. Yang, Singularly perturbed
critical Choquard equations, J. Differential Equations, 263(7)
(2017), 3943-3988.

\bibitem{Alves-Souto-Montenegro 2012} C.O. Alves, M.A. Souto, M. Montenegro, Existence of a ground
state solution for a nonlinear scalar field equation with critical
growth, Calc. Var. PDE, 43 (2012), 537-554.

\bibitem{Berestycki-Lions 1983}  H. Berestycki, P.L. Lions, Nonlinear scalar field
equations, I Existence of a ground state, Arch. Ration. Mech. Anal.,
82 (1983), 313-345.

\bibitem{Bogachev 2007} V.I. Bogachev, Measure theory, Springer, Berlin, 2007.


\bibitem{Brezis-Nirenberg 1983} H. Brezis, L. Nirenberg,  Positive solutions of nonlinear
elliptic equations involving critical Sobolev exponents, Commun.
Pure Appl. Math., 36 (1983), 437-477.

\bibitem{Cassani-Zhang 2016} D. Cassani, J. Zhang, Ground states and semiclassical states of
nonlinear Choquard equations involving Hardy-Littlewood-Sobolev
critical growth, arXiv preprint arXiv:1611.02919, (2016).


\bibitem{Gao-Yang-1} F. Gao, M. Yang, On the Brezis-Nirenberg type critical problem
for nonlinear Choquard equation, SCIENCE CHINA Mathematics, doi:
10.1007/s11425-000-0000-0.

\bibitem{Gao-Yang-2} F. Gao, M. Yang, On nonlocal Choquard equations with
Hardy-Littlewood-Sobolev critical exponents, Journal of Mathematical
Analysis and Applications, 448(2) (2017), 1006-1041.


\bibitem{Gross 1996} E.P. Gross, Physics of many-Particle systems, Gordon Breach,
New York, Vol.1, 1996.

\bibitem{Jeanjean 1999} L. Jeanjean, On the existence of bounded Palais-Smale sequences
and application to a Landesman-Lazer-type problem set on
$\mathbb{R}^N$, Proc Roy Soc Edinburgh, 129 (1999), 787-809.


\bibitem{Lieb 1977}  E.H. Lieb, Existence and uniqueness of the minimizing solution of
Choquard's nonlinear equation, Stud. Appl. Math., 57(2) (1977),
93-105.

\bibitem{Lieb-Loss 2001} E.H. Lieb, M. Loss, Analysis, volume 14 of graduate studies in
mathematics, American Mathematical Society, Providence, RI, (4)
2001.

\bibitem{Liu-Liao-Tang 2017} Jiu Liu, JiaFeng Liao, ChunLei Tang, Ground state solution for a
class of Schr\"{o}dinger equations involving general critical growth
term, Nonlinearity, 30 (2017), 899-911.

\bibitem{Liu-Liu-Wang 2013}  X.Q. Liu, J. Q. Liu, Z.Q. Wang,  Ground states for quasilinear
Schr\"{o}inger equations with critical growth, Calc. Var. PDE, 46
(2013), 641-669.

\bibitem{Moroz-Schaftingen JFA 2013} V. Moroz, J. Van Schaftingen, Groundstates of nonlinear Choquard equations: existence, qualitative properties and decay asymptotics, Journal of Functional Analysis, 265(2) (2013), 153-184.



\bibitem{Moroz-Schaftingen 2015} V. Moroz, J. Van Schaftingen, Existence of groundstates for a class of nonlinear Choquard
equations, Trans. Amer. Math. Soc., 367 (2015), 6557-6579.


\bibitem{Moroz-Schaftingen 2017} V. Moroz, J. Van Schaftingen, A guide to the Choquard equation, Journal of Fixed Point Theory and Applications, 19(1) (2017), 773-813.




\bibitem{Pekar 1954} S. Pekar, Untersuchung ¨¹ber die Elektronentheorie der
Kristalle, Akademie Verlag, Berlin, 1954.

\bibitem{Penrose 1996} R. Penrose, On gravity's role in quantum state reduction, Gen.
Rel. Grav., 28 (1996), 581-600.


\bibitem{Riesz1949AM} M. Riesz, L'int\'{e}grale de Riemann-Liouville et le probl\`{e}me de Cauchy, Acta Math., 81 (1949), 1-223.


\bibitem{Seok 2018} Jinmyoung Seok, Nonlinear Choquard equations: Doubly critical case,
Applied Mathematics Letters, 76 (2018), 148-156.


\bibitem{Strauss 1977} W.A. Strauss, Existence of solitary waves in higher dimensions,
Commun. Math. Phys., 55 (1977), 149-162.


\bibitem{Willem 1996} M. Willem, Minimax Theorems, Birkh\"{a}user, Boston, 1996.

\bibitem{Willem 2013} M. Willem,  Functional analysis: Fundamentals and
applications, Cornerstones, Vol.XIV, Birkh\"{a}user, Basel, 2013.


\bibitem{Zhang-Zou 2012} J.J. Zhang, W.M. Zou,  A Berestycki-Lions theorem
revisited, Commun. Contemp. Math., 14 (2012), 1250033.

\bibitem{Zhang-Zou 2014} J. Zhang, W.M. Zou, The critical case for a Berestycki-Lions
theorem, Sci. China Math., 57 (2014), 541-554.
\end{thebibliography}


\end{document}